\documentclass[a4paper,12pt]{amsart}
\usepackage{graphicx}
\usepackage{amssymb}
\usepackage{amsmath}
\usepackage{latexsym}
\usepackage{amsthm}
\usepackage{autograph,epic,latexsym,bezier,amsbsy,color,enumerate,amsfonts,amsmath,amscd,amssymb}
\usepackage[latin1]{inputenc}

\setloopdiam{17}\setprofcurve{10}

\newtheorem{defi}{Definition}[section]
\newtheorem{thm}[defi]{Theorem}

\newtheorem{lemma}[defi]{Lemma}
\newtheorem{example}[defi]{Example}
\newtheorem{remark}[defi]{Remark}

\begin{document}
\author{Alfredo Donno}
\address{Università degli Studi Niccolò Cusano - Telematica Roma - Via Don Carlo Gnocchi, 3 00166 Roma, Italia \qquad Tel.: +39 06 45678350
\qquad Fax: +39 06 45678379} \email{alfredo.donno@gmail.com,
alfredo.donno@unicusano.it}
\title{Generalized wreath products of graphs and groups}


\begin{abstract}
Inspired by the definition of generalized wreath product of
permutation groups, we define the generalized wreath product of
graphs, containing the classical Cartesian and wreath product of
graphs as particular cases. We prove that the generalized wreath
product of Cayley graphs of finite groups is the Cayley graph of
the generalized wreath product of the corresponding groups.
\end{abstract}

\maketitle

\begin{center}
\end{center}

\section{Introduction}

The idea of constructing new graphs starting from smaller
component graphs is very natural. Products of graphs were widely
studied in the literature for their theoretical interest in
Combinatorics, Probability, Harmonic Analysis, but also for their
practical applications. Standard products include the Cartesian
product, direct product, strong product, lexicographic product
\cite{imrich, sabidussi, cartesian} (see also the beautiful
handbook \cite{imrichbook}). In \cite{annals}, the zig-zag product
was introduced in order to produce constant-degree expanders of
arbitrary size (see the surveys \cite{expander, lubotullio} for
definition, properties and further references on expander graphs).
The zig-zag product and the simpler replacement product play also
an important role in Geometric Group Theory, since it turns out
that, when applied to Cayley graphs of two finite groups, they
provide the Cayley graph of the semidirect product of these groups
\cite{groups, alfredo1, expander, communications}, with a suitable
choice of the corresponding generating sets. An analogous result
holds for the classical wreath product of graphs (Theorem
\ref{proofwreath}, Section \ref{section2}).\\
\indent Inspired by the paper \cite{bayleygeneralized}, where the
definition of generalized wreath product of permutation groups is
given as a generalization of the classical direct and wreath
product of permutation groups, we define the {\it generalized
wreath product of graphs} (note that in \cite{erschler} a
different notion of generalized wreath product of graphs is
presented). It is remarkable that, with a particular choice of the
generating sets, our construction of the generalized wreath
product applied to Cayley graphs of finite groups gives the Cayley
graph of the generalized wreath product of the groups (Theorem
\ref{theoremlast}, Section \ref{section3}), providing a strong
generalization of Theorem \ref{proofwreath}.

\section{Preliminaries}\label{section2}

Let us start by recalling the definition of Cayley graph of a
finitely generated group with respect to some symmetric generating
set. We denote by $1_G$ the identity element of a group $G$.
\begin{defi}
Let $G$ be a group generated by a finite set $S$, and suppose that
$S$ is symmetric, i.e., if $s\in S$, then also $s^{-1}\in S$, and
that $1_G \not\in S$. The {\it Cayley graph} $Cay(G,S)$ of $G$
with respect to $S$ is the graph whose vertex set is $G$, and
where two vertices $g$ and $g'$ are adjacent (we will use the
notation $g\sim g'$) if there exists a generator $s\in S$ such
that $gs=g'$. The graph $Cay(G,S)$ is clearly a connected regular
graph of degree $|S|$.
\end{defi}
Note that we assume $1_G \not\in S$ in order to avoid loops in the
graph $Cay(G,S)$.

Let us recall the definition of Cartesian product (see, for
instance, \cite{cartesian}, or \cite{imrichbook, imrich}, where a
more general construction containing it as a particular case is
introduced) and wreath product of graphs \cite{erschler}.

\begin{defi}\label{definitioncartesianproduct}
Let $\mathcal{G}_1=(V_1, E_1)$ and $\mathcal{G}_2=(V_2,E_2)$ be
two finite graphs. The \textit{Cartesian product} $\mathcal{G}_1
\square \mathcal{G}_2$ is the graph with vertex set $V_1\times
V_2$, where two vertices $(v_1,v_2)$ and $(w_1,w_2)$ are adjacent
if:
\begin{enumerate}
\item either $v_1= w_1$ and $v_2\sim w_2$ in $\mathcal{G}_2$;
\item or $v_2=w_2$ and $v_1\sim w_1$ in $\mathcal{G}_1$.
\end{enumerate}
\end{defi}
It follows from the definition that if $\mathcal{G}_1$ is a
$d_1$-regular graph on $n_1$ vertices and $\mathcal{G}_2$ is a
$d_2$-regular graph on $n_2$ vertices, then the graph
$\mathcal{G}_1\square \mathcal{G}_2$ is a $(d_1+d_2)$-regular
graph on $n_1n_2$ vertices. Notice also that the graphs
$\mathcal{G}_1\square \mathcal{G}_2$ and $\mathcal{G}_2\square
\mathcal{G}_1$ are isomorphic.
\begin{defi}\label{defierschler}
Let $\mathcal{G}_1=(V_1, E_1)$ and $\mathcal{G}_2=(V_2,E_2)$ be
two finite graphs. The \textit{wreath product} $\mathcal{G}_1\wr
\mathcal{G}_2$ is the graph with vertex set $V_2^{V_1}\times V_1=
\{(f,v) | f:V_1\to V_2, \ v\in V_1\}$, where two vertices $(f,v)$
and $(f',v')$ are connected by an edge if:
\begin{enumerate}
\item ({\it edges of the first type}) either $v=v'=:\overline{v}$ and $f(w)=f'(w)$ for every $w\neq \overline{v}$,
and $f(\overline{v})\sim f'(\overline{v})$ in $\mathcal{G}_2$;
\item ({\it edges of the second type}) or $f(w)=f'(w)$, for every $w\in
V_1$, and $v\sim v'$ in $\mathcal{G}_1$.
\end{enumerate}
\end{defi}
It follows from the definition that, if $\mathcal{G}_1$ is a
regular graph on $n_1$ vertices with degree $d_1$ and
$\mathcal{G}_2$ is regular graph on $n_2$ vertices with degree
$d_2$, then the graph $\mathcal{G}_1\wr \mathcal{G}_2$ is a
$(d_1+d_2)$-regular graph on $n_1\cdot n_2^{n_1}$ vertices.

The wreath product of graphs represents a graph-analogue of the
classical wreath product of groups (Theorem \ref{proofwreath}). To
show that, we need to recall the basic definition of semidirect
product of groups. Let $A$ and $B$ be two finite groups, and
suppose that an action by automorphisms of $B$ on $A$ is defined,
i.e., there exists a group homomorphism $\phi: B \to Aut(A)$. For
every $a\in A$ and $b\in B$, we denote by $a^b$ the image of $a$
under the action of $\phi(b)$ and, similarly, we denote by $a^B =
\{a^b\ |\ b\in B\}$ the orbit of $a$ under the action of the group
$B$.
\begin{defi}
The \textit{semidirect product} $A\rtimes B$ is the group whose
underlying set is $A\times B = \{(a,b) \ |\ a\in A,b\in B\}$, and
whose group operation is defined by
$$
(a_1,b_1)(a_2,b_2) = (a_1a_2^{b_1},b_1b_2), \qquad \mbox{for all }
a_1,a_2\in A, b_1,b_2\in B.
$$
\end{defi}
It is easy to check that the identity of $A\rtimes B$ is given by
$(1_A,1_B)$, where $1_A$ and $1_B$ are the identity in $A$ and
$B$, respectively, and that $ (a,b)^{-1} =
((a^{-1})^{b^{-1}},b^{-1})$, for all $a\in A, b\in B$. Note that
the subgroup $A\times \{1_B\}$ of $A\rtimes B$ is isomorphic to
$A$, it is normal in $A\rtimes B$ and the action of $B$ on $A$ by
conjugation coincides with the original action of $B$ on $A$. In
formulas, we have $ (1_A,b)(a,1_B)(1_A,b)^{-1} = (a^b,1_B)$, for
all $a\in A, b\in B$.
\begin{defi}
Let $A$ and $B$ be two finite groups. The set $B^A = \{f:A\to B\}$
can be endowed with a group structure with respect to the
pointwise multiplication: $(f_1f_2)(a) = f_1(a) f_2(a)$. The {\it
wreath product} $A\wr B$ is the semidirect product $B^A\rtimes A$,
where $A$ acts on $B^A$ by shifts, i.e., if $f\in B^A$, one has
$$
f^a(x) = f(a^{-1}x), \quad \mbox{for all } a,x\in A.
$$
\end{defi}
We introduce some notation. If $f\in B^A$ and $A = \{a_1, a_2,
\ldots, a_{n_A}\}$, then we write $f=(f_1,f_2, \ldots, f_{n_A})$,
where we denote by $f_i$ the element $f(a_i)\in B$, for each
$i=1,\ldots, n_A$. In particular, an element of $B^A\times A$ will
be written as $((f_1,\ldots, f_{n_A}), a)$.
\begin{thm}\label{proofwreath}
Let $A$ and $B$ be two finite groups and let $S_A$ and $S_B$ be
symmetric generating sets for $A$ and $B$, respectively. Then
$$
Cay(A,S_A)\wr Cay(B,S_B) = Cay(A\wr B, S),
$$
where $S$ is the generating set of $A\wr B$ given by
$$
S = \{((s_b,1_B,\ldots,1_B),1_A), ((1_B,\ldots, 1_B),s_a)\ |\
s_a\in S_A, s_b\in S_B\}.
$$
\end{thm}
\begin{proof}
It is easy to check that $S$ is a symmetric generating set of
$A\wr B = B^A\rtimes A$. More precisely, any element $(f,a)=((f_1,
\ldots, f_{n_A}), a)\in B^A\rtimes A$, with $f_i\in B$ and $a\in
A$, can be decomposed as
$$
((f_1,1_B, \ldots, 1_B),1_A)((1_B, f_2,1_B,\ldots, 1_B),1_A)\cdots
((1_B, \ldots, 1_B,f_{n_A}),1_A)((1_B,\ldots,1_B),a).
$$
Now if $a = \prod_{k=1}^r s_k^{m_k}$, with $s_k\in S_A$ and
$m_k\in \mathbb{N}$, then one has
$$
((1_B,\ldots,1_B),a) = \prod_{k=1}^r ((1_B,\ldots,1_B),s_k)^{m_k}.
$$
Similarly, if $f_1 = \prod_{h=1}^{l}s_h^{m_h}$, with $s_h\in S_B$
and $m_h\in \mathbb{N}$, it holds:
$$
((f_1,1_B, \ldots, 1_B),1_A) = \prod_{h=1}^l ((s_h,1_B, \ldots,
1_B),1_A)^{m_h}.
$$
Finally, we have
$$
((1_B,\ldots, 1_B,f_i,1_B,\ldots,1_B),1_A) =
((1_B,\ldots,1_B),a_ia_1^{-1})((f_i,1_B,\ldots,1_B),a_1a_i^{-1})
$$
and so we conclude that $S$ generates $A\wr B$.\\
\indent What we have to prove now is that an edge in the graph
product $Cay(A,S_A)\wr Cay(B,S_B)$ corresponds exactly to the
multiplication by an element of $S$ in $A\wr B$.\\
\indent Consider an edge of the first type in $Cay(A,S_A)\wr
Cay(B,S_B)$: such an edge connects the vertices $((f_1,
\ldots,f_k, \ldots, f_{n_A}),a_k)$ and $((g_1, \ldots,g_k,\ldots,
g_{n_A}), a_k)$. By definition, it must be $f_i = g_i$ for every
$i\neq k$, whereas $f_k$ and $g_k$ are vertices adjacent in
$Cay(B,S_B)$. It follows that there exists $s_k\in S_B$ such that
$f_k s_k = g_k$. Then one gets $((g_1, \ldots,g_k,\ldots,
g_{n_A}), a_k)$ from $((f_1, \ldots,f_k, \ldots, f_{n_A}),a_k)$ by
multiplying by $((s_k,1_B,\ldots,1_B),1_A)$. In fact:
$$
((f_1,\! \ldots\!,\!f_k,\! \ldots\!,\!
f_{n_A}\!),\!a_k)\cdot((s_k,\!1_B,\!\ldots\!,\!1_B),\!1_A\!)\! =\!
((f_1,\! \ldots\!,\!f_k,\! \ldots\!,\!
f_{n_A})\cdot(s_k,\!1_B,\!\ldots\!,\!1_B)^{a_k}\!,\! a_k\cdot 1_A)
$$
$$
=((f_1,\! \ldots\!,\!f_k,\! \ldots\!,
f_{n_A})\cdot(\!1_B,\!\ldots\!,\!
1_B,\!\underbrace{s_k}_{k\tiny{\mbox{-th
place}}}\!,\!1_B,\!\ldots\!,\!1_B), a_k) \!=\! ((g_1, g_2,\!
\ldots\!,g_k,\!\ldots\!, g_{n_A}),a_k).
$$
This implies that edges of the first type correspond to
multiplication by elements of the form
$((s_b,1_B,\ldots,1_B),1_A)$, with $s_b\in S_B$.\\
\indent Consider now an edge of the second type in $Cay(A,S_A)\wr
Cay(B,S_B)$: such an edge connects the vertices $((f_1, \ldots,
f_{n_A}),a_k)$ and $((g_1, \ldots, g_{n_A}), a_h)$. By definition,
it must be $f_i = g_i$ for every $i=1,\ldots, n_A$, whereas $a_k$
and $a_h$ are vertices adjacent in $Cay(A,S_A)$. It follows that
there exists $s_a\in S_A$ such that $a_k s_a = a_h$. Then one gets
$((g_1, \ldots, g_{n_A}), a_h)$ from $((f_1, \ldots,
f_{n_A}),a_k)$ by multiplying by $((1_B,\ldots,1_B),s_a)$. In
fact:
$$
((f_1,\ldots, f_{n_A}),a_k)\cdot((1_B,\ldots,1_B),s_a) =
((f_1,\ldots, f_{n_A})\cdot(1_B,\ldots,1_B)^{a_k}, a_k\cdot s_a)=
$$
$$
((f_1, \ldots, f_{n_A})\cdot(1_B,\ldots, 1_B), a_h) =
((g_1,\ldots, g_{n_A}),a_h).
$$
This ensures that edges of the second type correspond to
multiplication by elements of the form $((1_B,\ldots,1_B),s_a)$,
with $s_a\in S_A$.
\end{proof}
\begin{example}\rm
Consider the graphs $\mathcal{G}_1$ and $\mathcal{G}_2$ in Fig.
\ref{duek2}.
\begin{center}
\begin{figure}[h!]
\begin{picture}(200,20)\unitlength=0,15mm
\letvertex A=(50,20)\letvertex B=(150,20)\letvertex C=(250,20)\letvertex D=(350,20)

\drawvertex(A){$\bullet$}\drawvertex(B){$\bullet$}\drawvertex(C){$\bullet$}\drawvertex(D){$\bullet$}

\drawundirectededge(A,B){}\drawundirectededge(C,D){}

\put(42,30){$u_1$}\put(142,30){$v_1$}\put(242,30){$u_2$}\put(342,30){$v_2$}
\put(95,-3){$\mathcal{G}_1$}\put(295,-3){$\mathcal{G}_2$}
\end{picture}\caption{}\label{duek2}
\end{figure}
\end{center}
Then the wreath product $\mathcal{G}_1\wr \mathcal{G}_2$ is the
octagonal graph in Fig. \ref{ottagono}.
\begin{center}
\begin{figure}[h]
\begin{picture}(200,180)\unitlength=0,11mm
\letvertex A=(250,420)\letvertex B=(390,360)\letvertex C=(450,220)\letvertex D=(390,80)\letvertex E=(250,20)\letvertex F=(110,80)\letvertex G=(50,220)
\letvertex H=(110,360)

\drawvertex(A){$\bullet$}\drawvertex(B){$\bullet$}\drawvertex(C){$\bullet$}\drawvertex(D){$\bullet$}
\drawvertex(E){$\bullet$}\drawvertex(F){$\bullet$}\drawvertex(G){$\bullet$}\drawvertex(H){$\bullet$}

\drawundirectededge(A,B){}\drawundirectededge(B,C){}\drawundirectededge(C,D){}\drawundirectededge(D,E){}
\drawundirectededge(E,F){}\drawundirectededge(F,G){}\drawundirectededge(G,H){}\drawundirectededge(H,A){}

\put(750,200){$\mathcal{G}_1\wr \mathcal{G}_2$}

\footnotesize
\put(170,440){$((u_2,u_2),u_1)$}\put(-78,350){$((v_2,u_2),u_1)$}

\put(-135,215){$((v_2,u_2),v_1)$}\put(-70,70){$((v_2,v_2),v_1)$}

\put(170,-15){$((v_2,v_2),u_1)$}\put(405,70){$((u_2,v_2),u_1)$}

\put(460,205){$((u_2,v_2),v_1)$}\put(405,355){$((u_2,u_2),v_1)$}
\end{picture}\caption{}\label{ottagono}
\end{figure}
\end{center}
If we regard the graphs $\mathcal{G}_1$ and $\mathcal{G}_2$ as the
Cayley graphs of two cyclic groups of two elements, then the
generators given by Theorem \ref{proofwreath} have order $2$, but
they do not commute: the wreath product
$\mathcal{G}_1\wr\mathcal{G}_2$ is the Cayley graph of the wreath
product of these groups, which is isomorphic to the dihedral group
of $8$ elements.

Now let $\mathcal{G}_3$ be a triangular graph. In Fig.
\ref{grande24} the wreath product $\mathcal{G}_3\wr \mathcal{G}_1$
is represented.
\begin{center}
\begin{figure}[h]
\begin{picture}(350,160)\unitlength=0,11mm

\letvertex AA=(10,200)\letvertex BB=(90,200)\letvertex CC=(50,271)
\drawvertex(AA){$\bullet$}\drawvertex(BB){$\bullet$}\drawvertex(CC){$\bullet$}

\drawundirectededge(AA,BB){}\drawundirectededge(BB,CC){}\drawundirectededge(CC,AA){}

\put(-50,230){$\mathcal{G}_3$}\put(750,230){$\mathcal{G}_3\wr
\mathcal{G}_1$}

\letvertex A=(310,340)\letvertex B=(310,60)\letvertex C=(360,10)\letvertex D=(640,10)\letvertex E=(690,60)\letvertex F=(690,340)\letvertex G=(640,390)
\letvertex H=(360,390)\letvertex I=(360,340)\letvertex L=(360,60)\letvertex M=(640,60)\letvertex N=(640,340)\letvertex O=(420,280)
\letvertex P=(420,230)\letvertex Q=(420,170)\letvertex R=(420,120)\letvertex S=(470,120)\letvertex T=(530,120)\letvertex U=(580,120)
\letvertex V=(580,170)\letvertex Z=(580,230)\letvertex X=(580,280)\letvertex Y=(530,280)\letvertex W=(470,280)


\drawvertex(A){$\bullet$}\drawvertex(B){$\bullet$}\drawvertex(C){$\bullet$}\drawvertex(D){$\bullet$}
\drawvertex(E){$\bullet$}\drawvertex(F){$\bullet$}\drawvertex(G){$\bullet$}\drawvertex(I){$\bullet$}\drawvertex(M){$\bullet$}\drawvertex(N){$\bullet$}
\drawvertex(H){$\bullet$}\drawvertex(L){$\bullet$}\drawvertex(O){$\bullet$}\drawvertex(T){$\bullet$}
\drawvertex(P){$\bullet$}\drawvertex(U){$\bullet$}\drawvertex(Q){$\bullet$}\drawvertex(V){$\bullet$}
\drawvertex(R){$\bullet$}\drawvertex(Z){$\bullet$}\drawvertex(Y){$\bullet$}\drawvertex(Z){$\bullet$}
\drawvertex(S){$\bullet$}\drawvertex(X){$\bullet$}\drawvertex(W){$\bullet$}

\drawundirectededge(A,B){}\drawundirectededge(A,H){}\drawundirectededge(A,I){}\drawundirectededge(B,L){}
\drawundirectededge(B,C){}\drawundirectededge(C,L){}\drawundirectededge(C,D){}\drawundirectededge(D,E){}
\drawundirectededge(D,M){}\drawundirectededge(E,M){}\drawundirectededge(E,F){}
\drawundirectededge(F,G){}\drawundirectededge(F,N){}\drawundirectededge(G,N){}\drawundirectededge(G,H){}
\drawundirectededge(H,I){}\drawundirectededge(I,O){}\drawundirectededge(O,P){}\drawundirectededge(O,W){}\drawundirectededge(P,Q){}
\drawundirectededge(Q,R){}\drawundirectededge(Q,S){}\drawundirectededge(R,L){}
\drawundirectededge(R,S){}\drawundirectededge(S,T){}\drawundirectededge(T,U){}\drawundirectededge(T,V){}
\drawundirectededge(U,M){}\drawundirectededge(U,V){}\drawundirectededge(V,Z){}\drawundirectededge(Z,Y){}\drawundirectededge(Z,X){}
\drawundirectededge(N,X){}\drawundirectededge(X,Y){}\drawundirectededge(Y,W){}\drawundirectededge(P,W){}
\end{picture}\caption{}\label{grande24}
\end{figure}
\end{center}
Note that the graph $\mathcal{G}_3\wr \mathcal{G}_1$ is obtained
in \cite{alfredo1} as the replacement product of the
$3$-dimensional Hamming cube with a triangular graph.
\end{example}
\begin{remark}\rm
Let $\mathcal{G}_1=(V_1, E_1)$ and $\mathcal{G}_2=(V_2,E_2)$ be
two finite graphs. The \textit{lexicographic product}
$\mathcal{G}_1 \circ \mathcal{G}_2$ is the graph with vertex set
$V_1\times V_2$, where two vertices $(v_1,v_2)$ and $(w_1,w_2)$
are adjacent if:
\begin{enumerate}
\item either $v_1\sim w_1$ in $\mathcal{G}_1$;
\item or $v_1=w_1$ and $v_2\sim w_2$ in $\mathcal{G}_2$.
\end{enumerate}
It follows from the definition that if $\mathcal{G}_1$ is a
$d_1$-regular graph on $n_1$ vertices and $\mathcal{G}_2$ is a
$d_2$-regular graph on $n_2$ vertices, then the graph
$\mathcal{G}_1\circ \mathcal{G}_2$ is a $(d_1n_2+d_2)$-regular
graph
on $n_1n_2$ vertices.\\
\indent Sometimes the lexicographic product of graphs, whose
automorphism group contains the wreath product of the automorphism
groups of the factors, is called wreath product. It has nothing to
do with the wreath product of Definition \ref{defierschler}.
\end{remark}

\section{Generalized wreath product of graphs}\label{section3}

Before introducing the notion of generalized wreath product of
graphs (Definition \ref{defimine}), we recall the definition of
poset block structure and generalized wreath product of
permutation groups introduced in \cite{bayleygeneralized}. We will
follow the same notation for the action to the right presented
there. See also \cite{ischia2008, orthogonal, lumpability}, where
the Gelfand pairs associated with the action of a generalized
wreath product of groups on a poset block structure are studied,
in connection with Markov chain Theory.

Let $(I,\preceq)$ be a finite poset, with $|I| = n$. For every
$i\in I$, the following subsets of $I$ can be defined:
\begin{itemize}
\item $A(i)=\{j\in I : j \succ i\}$ and $A[i] = A(i) \sqcup \{i\}$;
\item $H(i)=\{j\in I : j \prec i\}$ and $H[i] = H(i) \sqcup \{i\}$.
\end{itemize}
A subset $J\subseteq I$ is said {\it ancestral} if, whenever $i
\succ j$ and $j\in J$, then $i\in J$. Note that by definition
$A(i)$ and $A[i]$ are ancestral, for each $i\in I$. The set $A(i)$
is called the ancestral set of $i$, whereas the set $H(i)$ is
called the hereditary set of $i$.\\
\indent For each $i\in I$, let $X_i$ be a finite set, with
$|X_i|\geq 2$. For $J\subseteq I$, put $X_J = \prod_{i\in J}X_i$.
In particular, we put $X = X_I$. If $K\subseteq J \subseteq I$,
let $\pi^J_K$ denote the natural projection from $X_J$ onto $X_K$.
In particular, we set $\pi_J = \pi^I_J$ and $x_J=x\pi_J$, for
every $x\in X$. Moreover, we will use $X^i$ for
$X_{A(i)} = \prod_{j\in A(i)}X_j$ and $\pi^i$ for $\pi_{A(i)}$.\\
\indent Let $\mathcal{A}$ be the set of ancestral subsets of $I$.
If $J\in \mathcal{A}$, then the equivalence relation $\sim_J$ on
$X$ is defined as
$$
x \sim_J y \quad \Longleftrightarrow \quad x_J = y_J,\qquad
\mbox{for } x,y \in X.
$$
\begin{defi}
A {\it poset block structure} is a pair $(X,\sim_{\mathcal{A}})$,
where
\begin{enumerate}
\item $X = \prod_{(I,\preceq)}X_i$, with $(I,\preceq)$ a
finite poset and $|X_i| \geq 2$, for each $i\in I$;
\item $\sim_{\mathcal{A}}$ denotes the set of equivalence relations on
$X$ defined by all the ancestral subsets of $I$.
\end{enumerate}
\end{defi}
For each $i\in I$, let $G_i$ be a permutation group on $X_i$ and
let $F_i$ be the set of all functions from $X^i$ into $G_i$. For
$J\subseteq I$, we put $F_J = \prod_{i\in J}F_i$ and set $F =
F_I$. An element of $F$ will be denoted $f = (f_i)_{i\in I}$, with
$f_i \in F_i$.

\begin{defi}
For each $f\in F$, the action of $f$ on $X$ is defined as follows:
if $x = (x_i)_{i\in I}\in X$, then
\begin{eqnarray*}
x f = y,\quad  \mbox{where }y = (y_i)_{i\in I}\in X \quad
\mbox{and }\ y_i = x_i(x\pi^i f_i), \quad \mbox{for each }i\in I.
\end{eqnarray*}
\end{defi}
It is easy to verify that this is a faithful action of $F$ on $X$,
i.e, if $xf = xg$ for every $x\in X$, then $f=g$. Therefore
$(F,X)$ is a permutation group, called the {\it generalized wreath
product of the permutation groups $(G_i,X_i)_{i\in I}$} and
denoted $\prod_{(I,\preceq)}(G_i,X_i)$.

\begin{defi}
An automorphism of a poset block structure
$(X,\sim_{\mathcal{A}})$ is a permutation $\sigma$ of $X$ such
that, for every equivalence relation $\sim_J$ in
$\sim_{\mathcal{A}}$,
$$
x \sim_J y \qquad \Longleftrightarrow \qquad (x \sigma)\sim_J (y
\sigma), \qquad \mbox{for all } x, y \in X.
$$
\end{defi}
The following fundamental results are proven in
\cite{bayleygeneralized}. We denote by $Sym(X_i)$ the symmetric
group acting on $X_i$.
\begin{thm}
The generalized wreath product of the permutation groups $(G_i,
X_i)_{i\in I}$ is transitive on $X$ if and only if $(G_i, X_i)$ is
transitive for each $i\in I$.
\end{thm}
\begin{thm}
Let $(X, \sim_{\mathcal{A}})$ be the poset block structure
associated with the poset $(I,\preceq)$. Let $F$ be the
generalized wreath product $\prod_{(I,\preceq)}Sym(X_i)$. Then $F$
is the automorphism group of $(X,\sim_{\mathcal{A}})$.
\end{thm}
Given $f\in F$ and $J\subset I$ ancestral, a map $f_J: X_J \to
X_J$ is defined such that $f \pi_J = \pi_J f_J$. The following
lemma holds.
\begin{lemma}[\cite{bayleygeneralized}]\label{bayleyproduct}
Let $f,h\in F$. Then $fh = t$, with
$$
t_i  = f_i\cdot f_{A(i)}h_i, \qquad \mbox{for every }i\in I,
$$
where the product of $f_i$ and $f_{A(i)}h_i$ is pointwise.
\end{lemma}
\begin{proof}
Let $x\in X$. We have:
\begin{eqnarray*}
(xfh)_i &=& (x f)_i (x f \pi^i h_i)\\
&=& x_i(x\pi^i f_i) (x \pi^i f_{A(i)}h_i)\\
&=& x_i(x\pi^i(f_i\cdot f_{A(i)}h_i))\\
&=& x_i(x\pi^it_i).
\end{eqnarray*}
\end{proof}

\begin{remark}\rm
If $(I,\preceq)$ is a finite poset, with $\preceq$ the identity
relation (Fig. \ref{figure13}), then the generalized wreath
product is the permutation direct product. In this case, we have
$A(i) = \emptyset$, for each $i\in I$, so that an element $f$ of
$F$ is given by $f = (f_i)_{i\in I}$, where the function $f_i$ is
identified with an element of $G_i$, so that its action on $x_i$
does not depend on any other coordinate of $x$.
\begin{figure}[h]
\begin{picture}(300,30)
\put(80,20){$\bullet$}\put(110,20){$\bullet$}\put(140,20){$\bullet$}\put(150,23){\circle*{1}}
\put(160,23){\circle*{1}}\put(170,23){\circle*{1}}\put(180,23){\circle*{1}}\put(190,23){\circle*{1}}
\put(200,23){\circle*{1}}\put(210,23){\circle*{1}}\put(220,20){$\bullet$}
\put(78,8){$1$}\put(110,8){$2$}\put(140,8){$3$}\put(220,8){$n$}
\end{picture}\caption{}\label{figure13}
\end{figure}
\end{remark}
\begin{remark}\rm
If $(I,\preceq)$ is a finite chain (Fig. \ref{figure14}), then the
generalized wreath product is the classical permutation wreath
product $(G_1,X_1)\wr(G_2,X_2)\wr \cdots \wr(G_n,X_n)$. In this
case, we have $A(i) = \{1,2,\ldots,i-1\}$, for each $i\in I$, so
that an element $f\in F$ is given by $f = (f_i)_{i\in I}$, with
$$
f_i:X_1 \times \cdots \times X_{i-1}\longrightarrow G_i
$$
In other words, the action of $f$ on $x_i$ depends on its \lq\lq
ancestral\rq\rq coordinates $x_1, \ldots, x_{i-1}$.
\begin{figure}[h]
\begin{picture}(250,110)
\letvertex A=(125,100)\letvertex B=(125,80)\letvertex C=(125,60)
\letvertex D=(125,55)\letvertex E=(125,50)\letvertex F=(125,45)
\letvertex G=(125,40)\letvertex H=(125,35)\letvertex I=(125,15)

\drawvertex(A){$\bullet$}\drawvertex(B){$\bullet$}\drawvertex(C){$\bullet$}\drawvertex(H){$\bullet$}
\drawvertex(I){$\bullet$}
\drawvertex(D){\circle*{1}}\drawvertex(E){\circle*{1}}\drawvertex(F){\circle*{1}}\drawvertex(G){\circle*{1}}
\drawundirectededge(A,B){}\drawundirectededge(B,C){}
\drawundirectededge(H,I){}
\put(129,97){$1$}\put(129,77){$2$}\put(129,57){$3$}\put(130,32){$n-1$}\put(129,13){$n$}
\end{picture}\caption{}\label{figure14}
\end{figure}
\end{remark}
Inspired by the definition of generalized wreath product of
permutation groups, we introduce here the notion of generalized
wreath product of graphs.\\
\indent Let $\mathcal{S},\mathcal{T}$ be two sets. Given two
functions $f,g:\mathcal{S}\longrightarrow\mathcal{T}$, we will use
the notation $f\equiv g$ to say that $f(x)=g(x)$ for every $x\in
\mathcal{S}$. Similarly, we will write $f\equiv g$ in $A\subset
\mathcal{S}$ to say that $f(x) = g(x)$ for every $x\in A$.
\begin{defi}\label{defimine}
Let $(I,\preceq )$ be a finite poset, with $|I|=n$, and let
$\mathcal{G}_i=(V_i,E_i)$ be a finite graph, for every $i\in I$.
The {\it generalized wreath product} of the graphs
$\{\mathcal{G}_i\}_{i\in I}$ is the graph $\mathcal{G}$ with
vertex set
$$
V_\mathcal{G} = \{(f_1,f_2, \ldots, f_n)\ |\ f_i : \prod_{j\in
A(i)}V_j\to V_i, \ \mbox{for each }i\in I\}
$$
and where two vertices $f=(f_1,f_2, \ldots, f_n)$ and $h=(h_1,h_2,
\ldots, h_n)$ are adjacent if there exists $i\in I$, with $A(i) =
\{i_1, \ldots, i_p\}$, such that:
\begin{enumerate}
\item $f_j\equiv h_j$, for every
$j\neq i$;
\item $f_i\equiv h_i$ in $\prod_{j\in A(i)}V_j\
\setminus\ \{(f_{i_1}, \ldots, f_{i_p})\}$, and
$f_i(f_{i_1},\ldots, f_{i_p})\sim h_i(f_{i_1},\ldots, f_{i_p})$ in
$V_i$.
\end{enumerate}
The elements $f_{i_l}$, for $l=1,\ldots, p$, are defined
recursively, starting from indices whose ancestral set in $(I,
\preceq)$ is empty; more precisely, they represent vertices of
$V_{i_l}$ obtained by evaluating the functions $f_{i_l}$ on
$(f_{j_1},\ldots, f_{j_m})$, where $A(i_l) = \{j_1,\ldots, j_m\}$,
and so on.
\end{defi}
In other words, $f=(f_1,\ldots, f_n)$ and $h=(h_1,\ldots, h_n)$
are adjacent if there exists $i\in I$ such that $f_j\equiv h_j$
for each $j\neq i$ and $f_i$ coincides with $h_i$, except when
evaluated on the $p$-tuple $(f_{i_1}, \ldots, f_{i_p})$, where
$A(i) =\{i_1, \ldots, i_p\}$. Note also that, if $A(i)=\emptyset$,
the condition $(2)$ means that it must be $f_i\sim h_i$ in
$\mathcal{G}_i$.

We have $|V_\mathcal{G}| = \prod_{i=1}^n |V_i|^{\Pi_{j\in
A(i)}|V_j|}$, where we put $\Pi_{j\in A(i)}|V_j|=1$ if
$A(i)=\emptyset$. Moreover, if $\mathcal{G}_i$ is a $d_i$-regular
graph for every $i\in I$, then $\mathcal{G}$ is a regular graph of
degree $\sum_{i\in I}d_i$.

\begin{example}\rm
Consider the case where $|I|=4$, with the poset $(I,\preceq)$ and
the graphs $\mathcal{G}_i$ represented in Fig. \ref{figure15}.
\begin{figure}[h]
\begin{picture}(300,60)
\letvertex A=(40,50)\letvertex B=(60,30)\letvertex C=(80,50)
\letvertex D=(60,0)\letvertex E=(140,30)\letvertex F=(170,30)
\letvertex G=(245,40)\letvertex H=(230,15)\letvertex I=(260,15)

\drawvertex(A){$\bullet$}\drawvertex(B){$\bullet$}\drawvertex(C){$\bullet$}\drawvertex(D){$\bullet$}
\drawvertex(E){$\bullet$}\drawvertex(F){$\bullet$}\drawvertex(G){$\bullet$}
\drawvertex(H){$\bullet$}\drawvertex(I){$\bullet$}

\put(0,30){$(I,\preceq)$}\put(115,50){$\mathcal{G}_1=\mathcal{G}_2=\mathcal{G}_3$}\put(280,30){$\mathcal{G}_4$}

\drawundirectededge(A,B){}\drawundirectededge(C,B){}\drawundirectededge(D,B){}
\drawundirectededge(E,F){}\drawundirectededge(G,H){}\drawundirectededge(H,I){}\drawundirectededge(I,G){}

\put(36,55){$1$}\put(76,55){$2$}\put(65,25){$3$}\put(65,-3){$4$}\put(137,16){$a$}\put(167,16){$b$}\put(242,45){$c$}\put(227,3){$d$}\put(257,3){$e$}
\end{picture}\caption{}\label{figure15}
\end{figure}

\noindent In this case, $A(1)=A(2)=\emptyset$, $A(3)=\{1,2\}$ and
$ A(4)=\{1,2,3\}$, so that
$$
V_\mathcal{G} = \left\{(f_1,f_2,f_3,f_4)\ |\ f_1\in V_1,\ f_2\in
V_2,\ f_3: \{a,b\}^2\to \{a,b\},\ f_4:\{a,b\}^3\to
\{c,d,e\}\right\}.
$$
The function $f_3$ will be represented as a $4$-tuple of elements
in $\{a,b\}$, whereas $f_4$ will be an $8$-tuple of elements in
$\{c,d,e\}$ (coordinates are ordered lexicographically). We have
$|V_\mathcal{G}| = 2\cdot 2\cdot 2^4\cdot 3^8$. Consider, for
instance, the vertex $f=(a,b,(b,b,a,a),(c,e,d,c,e,d,e,e))\in
V_\mathcal{G}$. Its $5$ neighbors in $\mathcal{G}$ are the
vertices:
\begin{enumerate}
\item $(b,b,(b,b,a,a),(c,e,d,c,e,d,e,e))$, since $f_1=a\sim b$ in
$\mathcal{G}_1$;
\item $(a,a,(b,b,a,a),(c,e,d,c,e,d,e,e))$, since $f_2=b\sim a$ in
$\mathcal{G}_2$;
\item $(a,b,(b,a,a,a),(c,e,d,c,e,d,e,e))$, since $f_3(f_1,f_2) = f_3(a,b) =b\sim a$ in
$\mathcal{G}_3$;
\item $(a,b,(b,b,a,a),(c,e,d,d,e,d,e,e))$ and $(a,b,(b,b,a,a),(c,e,d,e,e,d,e,e))$, since\\ $f_4(f_1,f_2, f_3(f_1,f_2)) = f_4(a,b,b) =c\sim d,e$ in
$\mathcal{G}_4$.
\end{enumerate}
\end{example}
\begin{remark}\rm
If $(I,\preceq)$ is the poset $(I,\preceq_1)$ (resp.
$(I,\preceq_2)$) in Fig. \ref{figure12}, one obtains the classical
Cartesian product of Definition \ref{definitioncartesianproduct}
(resp. the classical wreath product of Definition
\ref{defierschler}).
\begin{figure}[h]
\begin{picture}(300,45)
\letvertex A=(60,25)\letvertex B=(90,25) \letvertex C=(225,40)\letvertex D=(225,10)

\drawvertex(A){$\bullet$}\drawvertex(B){$\bullet$}\drawvertex(C){$\bullet$}\drawvertex(D){$\bullet$}

\put(-10,20){$(I,\preceq_1)$} \put(260,20){$(I,\preceq_2)$}
\put(57,11){$1$}\put(87,11){$2$}\put(230,37){$1$}\put(230,7){$2$}
\drawundirectededge(C,D){}
\end{picture}\caption{}\label{figure12}
\end{figure}
\end{remark}
\begin{remark}\rm
In \cite{generalizedcrested}, finite posets are used to define
products of finite Markov chains, called generalized crested
products, and to develop their spectral analysis. Notice that this
construction generalizes the crested products of Markov chains
introduced in \cite{crested}.
\end{remark}
We are going to prove that the generalized wreath product of
Cayley graphs of finite groups is the Cayley graph of the
generalized wreath product of the groups.

Let $(I,\preceq)$ be a finite poset, with $|I|=n$, and let $G_i$
be a finite group, for each $i\in I$. Let $S_i$ be a symmetric
generating set for $G_i$ and consider the Cayley graph
$\mathcal{G}_i=Cay(G_i,S_i)$. In order to see the correspondence,
we regard the group $G_i$ as a permutation group on itself, acting
on its elements by right multiplication (according with the
notation of \cite{bayleygeneralized}). Definition \ref{defimine}
can be reformulated as follows.
\begin{defi}\label{defirevisited}
Let $(I,\preceq )$ be a finite poset, with $|I|=n$, and let
$\mathcal{G}_i=Cay(G_i,S_i)$, where $G_i$ is a finite group and
$S_i$ is a symmetric generating set of $G_i$, for all $i\in I$. We
construct the graph $\mathcal{G}$ with vertex set
$$
V_{\mathcal{G}} = \{(f_1,f_2, \ldots, f_n)\ |\ f_i : \prod_{j\in
A(i)}G_j\to G_i, \ \mbox{for each }i\in I\}
$$
and where the vertices $(f_1,f_2, \ldots, f_n)$ and $(h_1,h_2,
\ldots, h_n)$ are adjacent if there exists $i\in I$, with $A(i) =
\{i_1, \ldots, i_p\}$, such that:
\begin{enumerate}
\item $f_j\equiv h_j$, for every
$j\neq i$;
\item $f_i\equiv h_i$ in $\prod_{j\in A(i)}G_j\
\setminus\ \{(f_{i_1}^{-1}, \ldots, f_{i_p}^{-1})\}$, and the
vertices $f_i(f_{i_1}^{-1},\ldots, f_{i_p}^{-1})$ and
$h_i(f_{i_1}^{-1},\ldots, f_{i_p}^{-1})$ are adjacent in
$\mathcal{G}_i$.
\end{enumerate}
Also in this case, the elements $f_{i_l}$, for $l=1,\ldots, p$,
are defined recursively, so that they represent elements of the
group $G_{i_l}$ obtained by evaluating the functions $f_{i_l}$ on
$(f_{j_1}^{-1},\ldots, f_{j_m}^{-1})$, where $A(i_l) =
\{j_1,\ldots, j_m\}$.
\end{defi}
The following theorem is a strong generalization of Theorem
\ref{proofwreath}.
\begin{thm}\label{theoremlast}
The generalized wreath product $\mathcal{G}$ of the graphs
$\{\mathcal{G}_i=Cay(G_i,S_i)\}_{i\in I}$ is the Cayley graph of
the generalized wreath product $G$ of the groups $\{G_i\}_{i\in
I}$, with respect to the generating set
$$
S = \{\overline{f_i}=({\bf 1}_1, \ldots, {\bf 1}_{i-1},
\overline{s_i}, {\bf 1}_{i+1}, \ldots, {\bf 1}_n),\ i\in I\},
$$
where $\overline{s_i}$ is a function taking the value $s_i\in S_i$
on $(1_{G_{1_1}}, \ldots, 1_{G_{i_p}})$, with $A(i) =
\{i_1,\ldots, i_p\}$, and the value $1_{G_i}$ elsewhere, whereas
${\bf 1}_q$ is the constant function taking the value $1_{G_q}$ on
$\prod_{u\in A(q)}G_u$, for each $q\neq i$.
\end{thm}

\begin{proof}
Generalizing the argument developed in the proof of Theorem
\ref{proofwreath}, one can check that $S$ is a generating set of
the group $G$. Hence, we have to show that an adjacency in the
generalized wreath product of the Cayley graphs can be obtained by
multiplication by an element of $S$.\\
\indent Suppose that the vertices $f = (f_1,f_2, \ldots, f_n)$ and
$h = (h_1,h_2, \ldots, h_n)$ are adjacent, i.e., there exists
$i\in I$ satisfying the conditions of Definition
\ref{defirevisited}. Without loss of generality, we can suppose
that
$$
A(i) = \{l_1< \cdots <l_r <k_1<\cdots <k_s\}, \qquad \mbox{with
}r+s=p,
$$
with $A(l_i) = \emptyset$, $A(k_j) = \{m_{j,1}, \ldots,
m_{j,t_j}\}$ and $A(k_j)\subset A(k_{j+1})$. In particular,
observe that it must be $t_j\leq p$ and $A(k_j)\subset A(i)$. By
definition, it must be $f_j\equiv h_j$, for $j\neq i$; moreover,
we have $f_i\equiv h_i$ in $\prod_{j\in A(i)}G_j\ \setminus\
\{(f_{i_1}^{-1}, \ldots, f_{i_p}^{-1})\}$, and there exists
$s_i\in S_i$ such that $f_i(f_{i_1}^{-1},\ldots, f_{i_p}^{-1}) =
g_i$ and $h_i(f_{i_1}^{-1},\ldots, f_{i_p}^{-1}) = g_is_i$. We
have to show the following identity:
\begin{eqnarray}\label{identity}
(f_1,f_2,\ldots, f_i, \ldots, f_n)\overline{f_i} =
(h_1,h_2,\ldots, h_i, \ldots, h_n).
\end{eqnarray}
The identity \eqref{identity} is clearly true for every coordinate
$j\neq i$. We use Lemma \ref{bayleyproduct} in order to verify it
for the coordinate $i$. Let $x=(x_1,\ldots,x_n)\in \prod_{i\in
I}G_i$. Observe that, for every $c=1,\ldots,r$, we can put
$f_{l_c} = g_{l_c}$, for some $g_{l_c}\in G_{l_c}$, since
$A(l_c)=\emptyset$.\\
\indent The action of $(f_1,f_2,\ldots, f_i, \ldots,
f_n)\overline{f_i}$ on $x_i$ is given by
$$
x_i\cdot(f_i(x_{l_1}, \ldots, x_{l_r}, x_{k_1}, \ldots,
x_{k_s}))\cdot
$$
$$
\cdot((x_{l_1}g_{l_1}, \ldots,
x_{l_r}g_{l_r},x_{k_1}(f_{k_1}(x_{m_{1,1}},\ldots,x_{m_{1,t_1}})),\ldots,
x_{k_s}(f_{k_s}(x_{m_{s,1}},\ldots,x_{m_{s,t_s}})))\overline{s_i}).
$$
Now if
$$
(x_{l_1},\!\ldots\!, x_{l_r},x_{k_1},\!\ldots\!,x_{k_s})\! =\!
(g_{l_1}^{-1}\!,\! \ldots\!, g_{l_r}^{-1}\!,
(f_{k_1}(f_{m_{1,1}}^{-1}\!,\!\ldots\!,f_{m_{1,t_1}}^{-1}))^{-1}\!,\!\ldots\!,
(f_{k_s}(f_{m_{s,1}}^{-1}\!,\!\ldots\!,f_{m_{s,t_s}}^{-1}))^{-1})\!,
$$
then the argument of $\overline{s_i}$ is $(1_{G_{l_1}}, \ldots,
1_{G_{l_r}}, 1_{G_{k_1}}, \ldots, 1_{G_{k_s}})$, so that one gets
$$
x_i(f_1,f_2,\ldots, f_i, \ldots, f_n)\overline{f_i}= x_ig_is_i
$$
and so $(f \overline{f_i})_i = h_i$.\\
\indent Otherwise, the argument of $\overline{s_i}$ is different
from $(1_{G_{l_1}}, \ldots, 1_{G_{l_r}}, 1_{G_{k_1}}, \ldots,
1_{G_{k_s}})$, so that one gets $x_i(f_1,f_2,\ldots, f_i, \ldots,
f_n)\overline{f_i} = x_ig_i1_{G_i}=x_ig_i$ and so $(f
\overline{f_i})_i = f_i= h_i$.
\end{proof}

\section*{Acknowledgement}
I would like to express my deepest gratitude to Fabio Scarabotti
and Tullio Ceccherini-Silberstein for their continuous
encouragement. A part of this work was developed during my stay at
the Technische Universit\"{a}t of Graz, and I want to thank
Wolfgang Woess and Franz Lehner for several useful discussions.
This research was partially supported by the European Science
Foundation (Research Project RGLIS 4915).\\ \indent I would like
to thank the anonymous reviewers for their valuable comments and
suggestions to improve the quality of the paper.


\end{document}